\documentclass{mcma_v02_modif}
 
\usepackage{amsmath,amssymb,bbm}
\usepackage[mathscr]{euscript} 
\usepackage{array}
\usepackage{dsfont}

\newcommand{\indic}[1]{\ensuremath{ \mathbbm{1}_{#1} } }
\newcommand{\mbf}[1]{\ensuremath{ \mathbf{#1}} }
\newcommand{\mytilde}[1]{\tilde{ #1 }}
\usepackage{booktabs}

\title{Optimal  potential functions for the interacting particle system method} 
\headlinetitle{Optimal  potential functions}

\authorone{Hassane Chraibi}
\addressone{\'Electricit\'e de France (EDF), PERICLES department, 7 Boulevard Gaspard Monge, 91120 Palaiseau}
\countryone{France} 
\emailone{hassane.chraibi@edf.fr}

\authortwo{Anne Dutfoy}
\addresstwo{\'Electricit\'e de France (EDF), PERICLES department, 7 Boulevard Gaspard Monge, 91120 Palaiseau}
\countrytwo{France} 
\emailtwo{anne.dutfoy@edf.fr}

\authorthree{Thomas Galtier}
\addressthree{CMAP, \'Ecole polytechnique, Institut Polytechnique de Paris, 91128 Palaiseau Cedex
}
\countrythree{France}
\emailthree{thomas.galtier@polytechnique.edu}

\authorfour{Josselin Garnier}
\addressfour{CMAP, \'Ecole polytechnique, Institut Polytechnique de Paris,  91128 Palaiseau Cedex; INRIA
}
\countryfour{France}
\emailfour{josselin.garnier@polytechnique.edu}

\headlineauthor{H. Chraibi, A. Dutfoy, T. Galtier, and J. Garnier}

\abstract{The assessment of the probability of a rare event with a naive Monte-Carlo method is computationally intensive, so faster estimation or variance reduction methods are needed. We focus on one of these methods which is the interacting particle system (IPS) method. 
The method is not intrusive in the sense that the random Markov system under consideration is simulated with its original distribution, but 
selection steps are introduced that favor trajectories (particles) with high potential values. 
An unbiased estimator with reduced variance can then be proposed.
The method requires to specify a set of potential functions. The choice of these functions is crucial, because it determines the magnitude of the variance reduction. So far, little information was available on how to choose the  potential functions. This paper provides the expressions of the optimal potential functions minimizing the asymptotic variance of the estimator of the IPS method and it proposes recommendations for the practical design of the potential
functions.}

\keywords{Rare event simulation; interacting particle system; Sequential Monte-Carlo Samplers; Feymann-Kac particle filters}

\classification{MSC 2010 :
 65C05, 65C35, 60K35 65C40, 60J22}

\begin{document}

\section{Introduction}

The simulation of rare events in Markov systems has been intensively studied. 
Standard Monte Carlo simulations are usually prohibitive because too many simulations are needed to achieve the rare events. 
The general approach to speeding up such simulations is to favor the occurrence of the rare events by using importance sampling (IS).
In the IS strategy the system is simulated using a new set of  input and/or transition probability distributions, and unbiased estimates are obtained by multiplying the simulation output by a likelihood ratio \cite{rubinstein81}. 
The tricky part in IS is to properly choose the biased distributions in order to get a significant variance reduction.
Another difficulty is to run simulations with the biased dynamics, which is not always possible with big numerical codes in particular in industrial contexts.
Other promising biasing Monte Carlo methods have been studied, such as sampling importance resampling (SIR), 
sequential Monte Carlo (SMC), and multicanonical Monte Carlo (MMC) \cite{kunsch01,doucet01,berg00}.
All these methods involve biasing the distributions of some input variables to probe the tail of the pdf of the output variable of interest. 
The biased distributions are not assumed to be known a priori, but this knowledge is accumulated during the iterations. 
However, all these methods are intrusive. They require to modify the
numerical codes or the experiments in the sense that the distributions of the input random variables
have to be changed. 

Non-intrusive techniques have been proposed, that do not require to modify the distributions of the input variables. 
The principle of the Interacting Particle System (IPS) method  is to substitute for the bias of the input random variables selection steps based 
on the output variable.
It was first introduced in \cite{del2004feynman}, and with an alternative formulation in \cite{del2006sequential}.
The algorithm simulates the Markov system with selection and mutation steps and 
it can give an unbiased estimator of the probability of the rare event \cite{del2005genealogical}.
The IPS methodology is closely related 
to a class of Monte Carlo acceptance/rejection simulation techniques introduced in the fifties and nowadays used 
in many applications \cite{doucet01}.

The IPS algorithm generates a system of particles  by  a sequence of selection  and mutation steps.
The selection steps of the IPS algorithm are determined by potential functions.
As we will see below, particles with low potential values are killed while those with high potential values are multiplied during the selection steps.
When these potential functions are properly tuned, the IPS method can yield an estimator with a smaller variance than the Monte-Carlo estimator.
The main purpose of the IPS method is to provide such a variance reduction without modifying the dynamics of the process. Indeed, as we will see below, the mutation steps of the IPS algorithm  are carried out with the original Markov dynamics.
If we can also modify the dynamics of the process, then it is theoretically possible to reach a zero variance
and it is practically possible to design an iterative algorithm that approximates the optimal dynamics \cite{guarniero2017iterated}. 
Here we want to consider the case where the dynamics of the process is left untouched, because it is not always possible to modify the dynamics of the system in complex numerical codes or in real experiments,
as seen for instance in \cite{garnierdelmo06} where the numerical code is imposed. 
We consider such a framework in this paper (as we will discuss in the conclusion
a particularly favorable context is the multifidelity Monte Carlo framework \cite{multifidelity}).
The tricky part in IPS is the choice of the potential functions that determine the selection steps \cite{del2005genealogical,garnierdelmo06,wouters2016rare,jacquemart2016tuning,morio2013optimisation}.
In this paper, we consider the IPS method and we determine the optimal potential functions.

Let $(Z_k)_{0 \leq k\leq n}$ be a Markov chain with values in the measurable spaces $(E_k,  \mathcal E_k)$, with initial law $\nu_0$, and with a kernel $Q_k$ such that for $k>0$ and for   any bounded measurable function $h:E_k\to \mathbb{R}$
\begin{equation}
\mathbb E [ h(Z_k)  |Z_{k-1} ] =\int_{E_k} h(z_k) Q_k(dz_k|Z_{k-1}).
\end{equation}
For any bounded measurable function  $h:E_0\times \cdots \times E_n\to \mathbb{R}$ we have:
\begin{align}
\mathbb E [ h(Z_0,\ldots,Z_n) ] 
= \int_{E_n \times \cdots \times E_0} h(z_0,\ldots,z_n) 
Q_n(dz_n|z_{n-1})  \cdots Q_1(dz_1|z_0) \nu_0(dz_0) .
\end{align}
 Let $\mbf Z_k=(Z_0, Z_1, \dots ,Z_k)$ be a trajectory of size $k$, and let $\mbf E_k= E_0\times E_{ 1} \times \dots \times E_k $ be the set of trajectories of size $k$ that we equip with the product $\sigma$-algebra $\boldsymbol{\mathcal E}_k=\mathcal E_0 \otimes \mathcal E_1\otimes\dots \otimes \mathcal E_k$. For $i<j$, when it is necessary to differentiate  the coordinates of the trajectories $\mbf z_i$ and $ \mbf z_j$ we write the coordinates $z_{i,k}$ for $k\leq i$ and $z_{j,k}$ for $k\leq j$ such that $\mbf z_i=(z_{i,0},z_{i,1},\dots,z_{i,i} )$ and $\mbf z_j=(z_{j,0},z_{j,1},\dots,z_{j,j} )$. We write $\mbf z_k=(z_0,\ldots,z_k)$ when there is no ambiguity.
We introduce the Markov chain of the trajectories  $(\mbf Z_k)_{0 \leq k\leq n}$   with  values in  the measurable spaces $(\mbf E_k,\boldsymbol{\mathcal E}_k)$, and with the transition kernels $M_k$ such that:
\begin{align}
M_k(d{\bf z}_k|{\bf z}_{k-1})  = \delta_{{\bf z}_{k-1}}\big( d (z_{k,0},\,\ldots \, , z_{k,k-1}) \big)
Q_k(dz_{k,k} |z_{k-1,k-1}).
\end{align}
For any bounded measurable  function $h:\mbf E_n\to \mathbb R $ we have 
\begin{align}
\nonumber
  p_h&=\mathbb E\left[h(\mbf Z_n)\right] \\
  &=\int_{\mbf E_n\times\dots\times\mbf E_0 }h(\mbf z_n) \prod^n_{k=1}M_k(d\mbf z_k|\mbf z_{k-1}) \nu_0(d\mbf z_0). 
\end{align}
Rare event analyses are often carried out for reliability assessment where the system under consideration is modeled by the Markov chain of the trajectories  $(\mbf Z_k)_{0 \leq k\leq n}$  and 
we want to assess the probability of failure of the system. The system failure is charcaterized by a region $D$  of $\mbf E_n$ and we want to assess the probability of failure which happens when $\mbf Z_n$ enters $D$. 
Typically the system is reliable, which means that probability of failure $\mathbb P(\mbf Z_n\in D)$ is small.
In order to assess the probability of failure, we can take $h(\mbf z_n)=\indic D(\mbf z_n)$ so that $p_h=\mathbb P(\mbf Z_n\in D)$ and use the IPS method to get an estimation $\hat p_h$ of $p_h$. Although the main application of the IPS method is reliability assessment where it relates to the case $h=\indic{D}$, the result of this paper will be presented in its most general form, where $h$ is an arbitrary bounded measurable function.

The IPS method provides an estimator $\hat p_h$ of $p_h$.
As we have said, the choice of the potential functions $G_k:\mbf E_k \to [0,+\infty)$ for  $k\leq n-1$ used in the selection steps is critical 
because it determines the variance of the IPS estimator $\hat{p}_h$. 
So far, little information has been provided on the forms of  efficient potential functions. 
The standard approach is to find the best potential functions within a set of parametric potential functions. 
As a consequence the efficiency of the  method strongly depends on the quality of the chosen parametric family. 
For instance, in \cite{del2005genealogical} (where $E_s=\mathbb{R}$) the authors obtain their best variance reduction by choosing (remember that we denote $\mbf z_k=(z_0,\ldots,z_k)$)
$$
G_k(\mbf z_k)=\frac{\exp{\left[-\lambda V(z_k) \right]}}{\exp{\left[-\lambda V(z_{k-1}) \right]}}   ,
$$
 where $\lambda$ is a positive tuning parameter, and the quantity $V(z)=a-z$ roughly measures the proximity of $z$ to the critical region that was of the form $D=[a,+\infty)$. 
 In \cite{wouters2016rare} is is shown that it seems better to take a time-dependent proximity function $V_k$ instead of $V$, yielding: 
 $$
 G_k(\mbf z_k)=\frac{\exp{\left[-\lambda V_k(z_k) \right]}}{\exp{\left[-\lambda  V_{k-1}(z_{k-1}) \right]}},
 $$ 
 where the quantities $V_k(z)$ are again measuring the proximity of $z$ to $D$. 
 Once the set of parametric potential functions is chosen, it is necessary to optimize the tuning parameters of the potentials. 
 Different methods have been proposed. 
 In \cite{jacquemart2016tuning}, an empirical heuristic algorithm is provided;  in \cite{morio2013optimisation}   
  a meta model of the variance is minimized; 
  in  \cite{del2005genealogical} a large deviation principle is used as a guide. 
  One other common option for the potential functions is the one done in
 splitting methods. Indeed the splitting method  can also be seen as a version of the IPS method \cite{cerou2005limit}. 
 In this method one wants to assess the probability that a random variable $Z $  belongs to  a subset $B_n$.  
 A succession of nested sets $E=B_0\supseteq B_1\supseteq B_2\supseteq \dots \supseteq B_n $ is chosen by the practitioner 
 or possibly chosen in an adaptive manner \cite{cerou2007adaptive}.  
 One considers a sequence of random variables $(Z_i)_{i=1,\ldots,n}$  
  such that the small probability  $\mathbb P(Z \in B_n )$ can be decomposed into a  product of conditional probabilities: 
  $\mathbb P(Z\in B_n )=\prod_{i=1}^n\mathbb P(Z_i\in B_i|Z_{i-1}\in B_{i-1} )$ 
    and   $\mathbb P(Z\in B_n )=\mathbb E\left[h(\mbf Z_n)\right]$ by setting $h(\mbf z_n)= \indic{ B_n}( z_n)$. 
    In this method  the potential functions are chosen of the following form 
 $$
 G_k(\mbf z_k)=\indic{B_k}(z_k).
 $$ 
 One usually optimizes the variance reduction within this family of potential functions by optimizing the choice of the sets $(B_k)_{k\leq n}$.
 
 In this paper we tackle the issue of the choice of the potential functions. Our contribution is to provide the expressions of the optimal potential functions that minimize  the asymptotic variance of the IPS estimator. By showing there is a positive minimum to the variance reduction we confirm the empirical result  reported in the literature: the  maximal variance reduction reachable with the IPS method is bounded from below. 
 These expressions will hopefully lead the practitioners to design more efficient potential functions, that are close to the optimal ones. We
 give a few comments and recommendations based on the expressions of the optimal potential functions to guide the user in the choice of the potential functions.\\
 
  The rest of the paper is organized as follows. Section \ref{sec:IPS} introduces the IPS method. Section \ref{sec:optiG} presents  the potential functions that minimize  the asymptotic variance  of the IPS estimator. Section \ref{sec:applica} presents three examples of applications. Finally Section \ref{sec:disc} discusses the implications of our results.\\

  In the rest of the paper we use the following notations: We denote by $\mathscr M(A)$ the set of bounded measurable functions on a set $A$. If $f $ is a bounded measurable function, and $\eta$  is a measure  we note $\eta(f)=\int f\,d\eta$. If $M$ is a Markovian kernel, we denote by $M(f)$  the function such that  $M(f)(x)=\int f(y)M(dy|x)$, and for a measure $\eta$, we denote by $\eta M$  the measure such that
 \begin{equation}
  \eta M(f)=\int \int f(y)M(dy|x)\eta(dx).
  \end{equation}

\section{The IPS method}
\label{sec:IPS}

\subsection{A Feynman-Kac model}

The IPS method relies on a Feynman-Kac model \cite{del2004feynman} which  is defined in this subsection.
Let $(G_s)_{0 \leq s <n}$ be a set of potential functions $G_s : {\bf E}_s \to [0,+\infty)$
such that $\mathbb{E} [\prod_{s=0}^{n-1} G_s({\mbf Z}_s)]>0$.
For each $k< n$, we define a {target} probability measure $\tilde{\eta}_k$ on $\left( \mbf E_k,\boldsymbol{\mathcal E}_k\right)$, such that
$\forall B\in \boldsymbol{\mathcal E}_k$:
\begin{align}
 \tilde{\eta}_k(B) &=\frac{\mathbb E\left[\indic{B}(\mbf Z_{k}) \prod_{s=0}^k G_s(\mbf Z_{s})\right]}{ \mathbb E\left[  \prod_{s=0}^k G_s(\mbf Z_{s})\right]}. \label{eq:target} 
\end{align} 
For each $0 \leq k\leq n-1 $, we define the {propagated} target probability measure $\eta_{k+1}$ on $\left( \mbf E_{k+1},\boldsymbol{\mathcal E}_{k+1}\right)$ such that $\eta_{k +1}=\tilde{\eta}_{k}M_{k}$ and $\eta_0=\tilde{\eta}_0$. We have
$\forall B\in \boldsymbol{\mathcal E}_{k+1}$:
\begin{align}
 \eta_{k+1}(B) &=\frac{\mathbb E\left[\indic{B}(\mbf Z_{k+1}) \prod_{s=0}^k G_s(\mbf Z_{s})\right]}{ \mathbb E\left[  \prod_{s=0}^k G_s(\mbf Z_{s})\right]}. \label{eq:proptarget} 
\end{align}  
 Let $\Psi_k$ be the application that transforms a measure $\eta$ defined on $\mbf E_{k}$ into a measure $\Psi_k(\eta)$ defined on $\mbf E_{k}$ and such that \begin{equation}
  \Psi_k(\eta)(f)=\frac{\int G_k(\mbf z)f(\mbf z) \eta(d\mbf z) }{\eta(G_k)}.
\end{equation}
We say that $\Psi_k(\eta)$ gives the selection of $\eta$ through the potential $G_k$. Notice that $\tilde{\eta}_k$ is the selection of $\eta_k$ as $\tilde{\eta}_k=\Psi_k(\eta_k)$. The target distributions can therefore be built according to the following successive selection  and propagation steps:
$$\eta_k\overset{\Psi_k}{\xrightarrow{\mbox{\hspace{1cm}}}} \tilde{\eta}_k\overset{.M_k}{\xrightarrow{\mbox{\hspace{1cm}}}} \eta_{k+1}.$$

We also define the associated unnormalized measures $\tilde{\gamma}_k$ and $\gamma_{k+1}$, such that for $f\in\mathscr M(\mbf E_{k})$:
\begin{equation}
\begin{array}{l}
\displaystyle
 \tilde{\gamma}_{k}(f)  = \mathbb E\left[f(\mbf Z_{k})\prod_{s=0}^{k} G_s(\mbf Z_{s})  \right] , \\
\displaystyle
  \tilde{\eta}_k(f) =\frac{\tilde{\gamma}_{k}(f)}{\tilde{\gamma}_{k}(1)},
  \end{array}
   \label{eq:targetProp1}
\end{equation} and for $f\in\mathscr M(\mbf E_{k+1})$:
\begin{equation}
\begin{array}{l}
\displaystyle
 \gamma_{k+1}(f) =\mathbb E\left[f(\mbf Z_{k+1})\prod_{s=0}^{k} G_s(\mbf Z_{s})  \right]  ,\\
\displaystyle
  \eta_{k+1}(f)=\frac{\gamma_{k+1}(f)}{\gamma_{k+1}(1)}. 
  \end{array}
  \label{eq:targetProp2}
\end{equation}
Denoting $f_h(\mbf z)=\frac{h( \mbf z )}{\prod_{s=0}^{n-1} G_s(\mbf z_s) }$ if $\prod_{s=0}^{n-1} G_s(\mbf z_s)>0$ and $f_h({\mbf z})=0$ otherwise,  we have:
\begin{equation}
  p_h=\gamma_n(f_h)= \eta_n(f_h) \overset{n-1}{\underset{{k=0}}{\prod}} \eta_k\big(G_k \big )\label{eq:ph}.
\end{equation}

\subsection{The IPS algorithm and its estimator}

The IPS method provides an algorithm to generate  samples whose weighted empirical measures approximate 
the probability measures $\eta_{k}$ and $ \tilde{\eta}_k$  for each step $k$. 
These approximations are then used to provide an estimator of $p_h$. For the sample approximating $ \eta_{k }$, 
we denote $\mbf Z_{k}^j$ the $j^{th}$ trajectory and  $ W_{k}^j$ its weight. Similarly, in the sample approximating $ \tilde{\eta}_k$, 
we denote $\mytilde{\mbf Z}_{k}^j$  the $j^{th}$ trajectory and $\mytilde{W}_{k}^j$ its associated weight. 
For simplicity reasons, in this paper, we consider that the samples all contain $ N$ trajectories, 
but it is possible to modify the sample size at each step, as illustrated in \cite{lee2018variance}. 
The empirical  measure approximating   $\eta_{k}$ and $ \tilde{\eta}_k$  are denoted by $\eta_{k}^N$ and $ \tilde{\eta}_k^N$ and are defined by:
\begin{equation}
  \tilde{\eta}_k^{N}  = \sum_{i=1}^{N} \,\mytilde W_k^i\, \delta_{ \mytilde{\mbf Z}_{k}^i}  \quad\mbox{and}\quad
  \eta_k^{N}  = \sum_{i=1}^{N} \,W_k^i\, \delta_{ \mbf Z_{k}^i} \,.\label{eq:estiEta}
\end{equation}
So for all  $f\in \mathscr M( \mbf E_{k})$,
\begin{equation}
  \tilde{\eta}_k^{N} (f) = \sum_{i=1}^{N} \,\mytilde{W}_k^i\, f \big( \mytilde{\mbf Z}_{k}^i \big) \quad\mbox{and} \quad
  \eta_k^{N} (f) = \sum_{i=1}^{N} \,W_k^i\, f \big( \mbf {Z}_{k}^i \big)\,.
  \label{eq:estiEta2}
\end{equation}
By plugging these estimators into equations \eqref{eq:targetProp1} and \eqref{eq:targetProp2}, we get estimators for the unnormalized mesures. Denoting by $\tilde{\gamma}_k^{N}$ and  $\gamma_k^{N}$ these estimators, for all $f\in \mathscr M( \mbf E_{k})$, we have:
\begin{equation}
  \tilde{\gamma}_k^{ N } (f) = \tilde{\eta}_k^{N }(f) \prod_{s=0}^{k-1} \eta_s^{N }(G_s)
\end{equation}
and
\begin{equation}
  \gamma_k^{ N } (f) =\eta_k^{N }(f)          \prod_{s=0}^{k-1} \eta_s^{N}(G_s).
  \label{eq:estigamma}
\end{equation}
In particular if we apply \eqref{eq:estigamma} to the test function $f_h$, we get an estimator $\hat p_h$ of $p_h$ defined by:
\begin{equation}
\hat p_h = \eta_{n}^{N}(f_h) \overset{n-1}{\underset{{k=0}}{\prod}}  \eta_k^{N}\big(G_k \big ) . \label{eq:hatp}
\end{equation}

\begin{table}
\fbox{
\begin{tabular}{l}
  {\it Initialization:}\\
  $ k=0, \ \forall j =1,\ldots,N, \ \mbf Z_0^j\overset{i.i.d.}{\sim} \eta_0$, \\
  $ W_0^j=\frac 1 N $, and $\mytilde W_{0}^j=\frac{ G_0( \mbf Z_{0}^j)}{\sum_{s=1}^N  G_0( \mbf Z_{0}^s)}$ \\
 {\bf while} {$k<n$} {\bf do}\\
 \quad {\it Selection:}\\
 \quad $ (\tilde N_k^j)_{j=1,\ldots,N }\sim Mult\big(  N , (\mytilde W_{k}^j)_{j=1,\ldots,N }\big)$\\
  \quad $\forall j =1,\ldots,N ,\ \mytilde W_{k}^j = \frac 1 { N }$\\
 \quad  {\it Propagation:}\\
 \quad {\bf for} 
 {$j:=1,\ldots,N$} {\bf do}\\
 \quad \quad using the kernel $M_k$, continue the trajectory $\mytilde{\mbf{Z}}_{k}^j$ to get $ \mbf{Z}_{k+1}^j$\\
\quad \quad 	set $W_{k+1}^j=\mytilde W_k^j$ and $\mytilde W_{k+1}^j=\frac{W_{k+1}^j G_{k+1}( \mbf Z_{k+1}^j)}{\sum_{s} W_{k+1}^s G_{k+1}( \mbf Z_{k+1}^s)}$\\
  \quad {\bf if}
  {$\forall j,\ \mytilde W_{k+1}^j=0$} 
  {\bf then} \\
 \quad \quad {$\forall q>k$,  set $ \eta_q^N=\tilde \eta_q^N=0$ and Stop }\\
\quad  {\bf else}\\
  \quad \quad   {$k=k+1$}
  \end{tabular}
  	}
\caption{IPS algorithm.}
\label{table:algo1}
\end{table} 

The IPS algorithm builds the samples sequentially, alternating between a selection step and a propagation step.
The $k^{th}$ selection step transforms the  sample $(\mbf Z_k^j,W_k^j)_{j\leq N}$, into the sample 
$(\mytilde{\mbf Z}_k^j,\mytilde{W}_k^j)_{j\leq N}$. This transformation is done with a multinomial resampling scheme. This means that the $\mytilde{\mbf Z}_k^j$'s are drawn with replacement from the sample $ (\mbf Z_k^j)_{j\leq N}$, each trajectory $ \mbf Z_k^j$ having a probability $ \frac{W_{k}^j G_k( \mbf Z_{k}^j)}{\sum_{i=1}^N W_{k}^i G_k( \mbf Z_{k}^i)}$ to be drawn each time.  We let $\tilde N_k^j$ be the number of times the particle $ \mbf Z_k^j$ is replicated in the sample $(\mytilde{\mbf Z}_k^j,\mytilde{W}_k^j)_j$, so $ N =\sum_{j=1}^{N } \tilde N_k^j$. After this resampling the weights $\mytilde W_k^j$ are set to $ 1/N$.
The interest of this selection by resampling is that it discards low potential trajectories and replicates high potential trajectories. Thus, the selected sample focuses on trajectories that will have a greater impact on the estimations of the next distributions once extended. \\
If one specifies potential functions that are not positive everywhere, there can be a possibility that at a step $k$ we get $\forall j, G_k(\mbf Z_{k}^j)=0$. When this is the case, the probability for resampling cannot be defined, the algorithm stops, and we consider that $\forall s\geq k$ the measures $\tilde{\eta}_s^N$ and $ \eta_{s+1}^N$ are equal to the null measure.\\

Then the $k^{th}$ propagation step transforms the  sample $(\mytilde{\mbf Z}_k^j,\mytilde{W}_k^j)_{j\leq N}$, into the sample 
$(\mbf Z_{k+1}^j,W_{k+1}^j)_{j\leq N}$. Each trajectory $ \mbf Z_{k+1}^j $ is obtained by extending the trajectory  $ \mytilde{\mbf Z}_k^j $ 
one step further using the transition kernel $M_k$. The weights satisfy $W_{k+1}^j=\mytilde{W}_k^j ,\ \forall j$. 
Then the procedure is iterated until the step $n$. The full algorithm to build the samples is displayed in table \ref{table:algo1}.\\

 For $k<n$, we denote by $\hat{\mbf E}_k=\{\mbf z_k\in \mbf E_k, G_k(\mbf z_k)>0\}$ the support of $G_k$, and we denote $ \hat{\mbf E}_n=\{\mbf z_n\in \mbf E_n, h(\mbf z_n)>0\}$ the support of $h$. We will make the following assumption on the potential functions:
 \begin{equation}
   \exists\,\varepsilon >0,\quad \forall k\leq n,\ \forall \mbf z_{k-1}\in \hat{\mbf E}_{k-1},\ M_{k-1}( \hat{ \mbf E}_{k}|\mbf z_{k-1})>\varepsilon \label{eq:(G)}\tag{G} .
 \end{equation} 

\begin{theorem}
 When the potential functions satisfy \eqref{eq:(G)}, $\hat p_h$ is unbiased and strongly consistent. \label{consitant}
\end{theorem}
The proof of theorem \ref{consitant} can be found in \cite{del2004feynman} chapter 7.
\begin{theorem} 
When the potential functions are positive-valued, we have
\begin{equation}
\sqrt N (\hat p_h- p_h)\overset{d}{\underset{N\to\infty}{\longrightarrow}} \mathcal N\big(0,\sigma^2_{IPS,G}\big),
\end{equation}
where
\begin{align}
\sigma^2_{IPS,G}& =\sum_{k=0}^{ n } \left\{\mathbb E \bigg[ \prod_{i=0}^{k-1} G_i (\mbf Z_{i})\bigg]  
\mathbb E \bigg[ \mathbb E[h(\mbf Z_{n})|\mbf Z_{k} ]^2\, \prod_{s=0}^{k-1} G_s^{-1 }(\textbf Z_{s})\bigg]-  p_h^2\right\},
\label{eq:var}
\end{align}
 with the convention $\overset{-1}{\underset{i=0}{\prod}}  G_i (\mbf Z_{i})=\overset{-1}{\underset{i=0}{\prod}} G_i^{-1} (\mbf Z_{i})=1$.
\end{theorem}
A proof of this Central Limit Theorem (CLT) can be found in \cite{del2004feynman} chapter 9 or \cite{del2005genealogical}. This CLT is an important result of the particle filter literature. The non-asymptotic fluctuations  of the particle filters have been studied in \cite{cerou2011nonasymptotic}. Recently two weakly consistent estimators of the asymptotic variance $\sigma^2_{IPS,G}$,  that are based on a single run of the method,  have been proposed in \cite{lee2018variance}. One of these estimators is closely related to the one that was proposed in \cite{chan2013general}. 
 
\section{The optimal potential functions}
 \label{sec:optiG}
Here we aim at estimating $p_h$.
The choice of the potential functions has an impact on the variance of the estimation, so we would like to find  potential functions that minimize  the asymptotic variance \eqref{eq:var}. This choice depends on the function $h$. Also, note that if potential functions $G_k$ and $G'_k$ are such that $G_k = a_k G'_k$ with $a_k>0$, then they yield the same variance: $\sigma^2_{IPS,G}=\sigma^2_{IPS,G'}$. 
If potential functions $G_k$ and $G'_k$ are such that $G_k (\mbf Z_k)= G'_k(\mbf Z_k)$  a.s., then they also yield the same variance.
Therefore all potential functions will be defined up to a multiplicative constant $\nu_k$ a.s., where $\nu_k$ is the distribution of $\mbf Z_k$.
 
\begin{theorem}
\label{thm3}
For $k=0,\ldots,n-1$, let $g_k :\mbf E_{k}\to [0,+\infty)$ be defined by:
\begin{equation}
\label{def:gk}
g_k(\mbf z_{k}) = \mathbb E  \Big[   \mathbb E\big[ h(\mbf Z_{n}) \big|\mbf Z_{k+1}\big]^2  \big|  \mbf Z_{k} = \mbf z_{k}  \Big]  .
\end{equation}
For $k=0,\ldots,n-1$, let $G^{*}_k :\mbf E_{k}\to  [0,+\infty)$ be defined by:
\begin{align}
 G^{*}_k(\mbf z_{k}) = \sqrt{ \frac{g_{k}(\mbf z_k)}{ g_{k-1}(\mbf z_{k-1})}}
 \label{eq:optiG1}
\end{align}
with $\mbf z_{k-1}=(z_{k,0},\ldots,z_{k,k-1})$, if $g_{k-1}(\mbf z_{k-1}) >0$,
and $ G^{*}_k(\mbf z_{k}) = 0 $ otherwise. We use the convention $g_{-1}=1$.

The potential functions minimizing $\sigma^2_{IPS,G}$ are the $G^*_k $'s, $k\leq n-1$. 
They are unique in the sense that they are defined up to a multiplicative constant $\nu_k$ a.s..

The optimal variance of the IPS method with $n$ steps is then 
\begin{align}
\sigma^2_{IPS,G^*} = &
 \mathbb E \left[  \,  \mathbb E\big[ h(\mbf Z_{n})\big|\mbf Z_{0}\big]^2 \right]-  p_h^2\nonumber\\
&+\sum_{k=1}^{ n } \left\{\mathbb E \left[\sqrt{\mathbb E\Big[ \  \mathbb E\big[ h(\mbf Z_{n})\big|\mbf Z_{k}\big]^2\big|\mbf Z_{k-1}\Big]}\right]^2-  p_h^2\right\}.
\label{eq:varOpti}
\end{align}

\end{theorem} 

\begin{proof}
 As we lack mathematical tools to minimize  $\sigma^2_{IPS,G}$ over the set of potential functions $(G_k)_{k < n}$, the strategy is the following one:
First we guess the expressions \eqref{eq:optiG1} by a heuristic (and greedy) argument in the spirit of  the one-step-ahead minimization method 
used in \cite{douc09,gerber19} for instance.
Second we provide the proof that \eqref{eq:optiG1} indeed minimizes the asymptotic variance of the IPS estimator.
 
 Assuming that we already know the $k-2$ first potential functions, we start by trying to find the $k-1$-th potential function $G_{k-1}$ that minimizes the $k$-th term of the sum in \eqref{eq:var}. This is equivalent to minimize  the quantity \begin{equation}
   \mathbb E \bigg[ \prod_{i=0}^{k-1} G_i (\mbf Z_{i})\bigg]  \mathbb E \bigg[\mathbb E \big[ \mathbb E[h(\mbf Z_{n})|\mbf Z_{k} ]^2\big|\mbf Z_{k-1} \big]\, \prod_{s=0}^{k-1} G_s^{-1 }(\textbf Z_{s})\bigg] \label{eq:16} 
 \end{equation}
 over $G_{k-1}$.
As the $G_{k-1}$ are equivalent up to a multiplicative constant, 
we simplify the equation by choosing a multiplicative constant so that 
$\mathbb E \bigg[ \prod_{i=0}^{k-1} G_i (\mbf Z_{i})\bigg] =1$. 
Our minimizing problem then becomes the minimization of \eqref{eq:16} under the constraint
 $ \mathbb E \bigg[ \prod_{i=0}^{k-1} G_i (\mbf Z_{i})\bigg]=1$. 
 In order to be able to use a Lagrangian minimization we temporarily assume that the distribution of $\mbf Z_{k-1}$
  is discrete and that $\mbf Z_{k-1}$  takes its values in a finite or numerable set $\textbf E_{k-1}$. 
  For $\textbf z\in \textbf E_{k-1}$, we denote $a_{\textbf z}=\mathbb P(\mbf Z_{k-1}=\textbf z)$,
  $ g_{\textbf z}= \mathbb E \big[ \mathbb E[h(\mbf Z_{n})|\mbf Z_{k} ]^2\big|\mbf Z_{k-1}={\textbf z} \big]$ 
and $d_{\textbf z}=\prod_{i=0}^{k-2} G_i (\mbf z_{i})G_{k-1} ({\textbf z})$ (with $\mbf z_i=(z_0,\ldots,z_i)$). 
The minimization problem becomes the minimization of 
\begin{equation}
 \mathcal L= \left(\sum_{ \textbf z \in \textbf E_{k-1} }\frac{p_{\textbf z} g_{\textbf z}}{d_{\textbf z}} \right) -\lambda \left(1-\sum_{{\textbf z}\in{\textbf E}_{k-1}}p_{\textbf z} d_{\textbf z} \right).
\end{equation} 
Finding the minimum of this Lagrangian  we get that $d_{\textbf z}=\frac{\sqrt{g_{\textbf z}}}{\sum_{{\textbf z^\prime}\in{\textbf E}_{k-1}}p_{\textbf z^\prime}\sqrt{g_{\textbf z^\prime}}}$.
Now relaxing the constraint of the multiplicative constant, we get that  
$$
\prod_{i=0}^{k-1} G_i (\mbf z_{i}) \propto \sqrt{\mathbb E \big[ \mathbb E[h(\mbf Z_{n})|\mbf Z_{k} ]^2\big|\mbf Z_{k-1} =\mbf z \big]} =\sqrt{g_{k-1}(\mbf z)},
$$
which gives the desired expressions. After these heuristic arguments we can now  rigorously check that these expressions, obtained by minimizing each of the term of sum in \eqref{eq:var} one by one, also minimize  the whole sum  for any distribution of the $\mbf Z_{k-1}$'s. \\

The proof now consists in showing that, for any set of positive-valued  potential functions $(G_s)$, we have $\sigma^2_{IPS,G}\geq \sigma^2_{IPS,G^*} $. This is done by bounding from below each term of the sum in \eqref{eq:var}. 

We denote by $supp  \prod_{s=0}^{k-1} G^{*}_s$, i.e., the support of the function $ \mathbf z_{k-1}\mapsto \prod_{s=0}^{k-1} G^{*}_s(\mathbf z_s) $,
with the convention $\mbf z_s=(z_{k-1,0},\ldots,z_{k-1,s})$ in the product. 
We start by decomposing a product of potential functions as follows: $  \forall k\in\{1,\dots,n\}$,
\begin{align}
\prod_{s=0}^{k-1} G_s(\mbf z_{s}) &= \epsilon_{k-1}(\mbf z_{k-1})\prod_{s=0}^{k-1} G^{*}_s(\mbf z_{s})  +\bar\epsilon_{k-1}(\mbf z_{k-1})  
\label{eq:bhjk} 
\end{align} 
where
 $$
 \epsilon_{k-1}(\mbf z_{k-1})=\dfrac{\prod_{s=0}^{k-1} G_s(\mbf z_{s})}{\prod_{s=0}^{k-1} G^{*}_s(\mbf z_{s})}\ \mbox{and}\ \bar\epsilon_{k-1}(\mbf z_{k-1})=0
 $$ 
 if $ \mbf z_{k-1}\in supp \prod_{s=0}^{k-1} G^{*}_s$, and
$$
\epsilon_{k-1}(\mbf z_{k-1})=0\ \mbox{and}\ \bar\epsilon_{k-1}(\mbf z_{k-1})=\prod_{s=0}^{k-1} G_s(\mbf z_{s}),
$$
if $ \mbf z_{k-1}\notin supp \prod_{s=0}^{k-1} G^{*}_s$.
Using \eqref{eq:bhjk}  we get that
\begin{align}
 \mathbb E &\bigg[ \ \prod_{s=0}^{k-1} G_s (\textbf Z_{s})\bigg]  \mathbb E \bigg[ \mathbb E[h(\textbf Z_{n})|\textbf Z_{k} ]^2\, \prod_{s=0}^{k-1} G_s^{-1 }(\textbf Z_{s})\bigg] \nonumber \\
\nonumber 
    = &\  \mathbb E \bigg[ \epsilon_{k-1}(\mbf Z_{k-1})\prod_{s=0}^{k-1} G^{*}_s (\textbf Z_{s})\bigg]
    \mathbb E \bigg[ g_{k-1}  ( \textbf Z_{k-1} ) \prod_{s=0}^{k-1} G_s^{-1 }(\textbf Z_{s})\bigg]  \nonumber \\
  \nonumber   & +  \mathbb E \bigg[ \bar\epsilon_{k-1}(\mbf Z_{k-1}) \bigg]
     \mathbb E \bigg[ g_{k-1}  ( \textbf Z_{k-1} ) \prod_{s=0}^{k-1} G_s^{-1 }(\textbf Z_{s})\bigg] \nonumber \\
 \geq 
  &\   \mathbb E \bigg[ \epsilon_{k-1}(\mbf Z_{k-1})\prod_{s=0}^{k-1} G^{*}_s (\textbf Z_{s})\bigg]
    \mathbb E \left[ \frac{g_{k-1}  ( \textbf Z_{k-1} )}{\prod_{s=0}^{k-1} G_s (\textbf Z_{s})}\right]+0  . \label{eq:njin}
\end{align}
For $ \mbf z_{k-1}\in supp \prod_{s=0}^{k-1} G^{*}_s$ we have: 
$$ 
\prod_{s=0}^{k-1} G^{*}_s(\mbf z_{s})\propto \sqrt{g_{k-1}( \mbf z_{k-1})} . 
$$
Consequently  $supp\, g_{k-1}  = supp \prod_{s=0}^{k-1} G^{*}_s$ 
and we get 
\begin{align} 
\nonumber
& \mathbb E \left[ \frac{g_{k-1}  ( \textbf Z_{k-1} ) }{\prod_{s=0}^{k-1} G_s (\textbf Z_{s})}\right] 
  =\mathbb E \left[ \frac{
 g_{k-1}  ( \textbf Z_{k-1} ) }{\epsilon_{k-1}(\mbf Z_{k-1})\prod_{s=0}^{k-1} G_s^{*} (\textbf Z_{s})}\right] \nonumber\\
 &=\mathbb E \left[ \frac{1} {\epsilon_{k-1}(\mbf Z_{k-1}) }\prod_{s=0}^{k-1} G_s^* (\textbf Z_{s})\right] \label{eq:dcvb}.
\end{align}
Combining \eqref{eq:dcvb} with inequality \eqref{eq:njin} we obtain that
\begin{align}
 \mathbb E &\bigg[ \ \prod_{s=0}^{k-1} G_s (\textbf Z_{s})\bigg]  \mathbb E \bigg[ \mathbb E[h(\textbf Z_{n})|\textbf Z_{k} ]^2\, \prod_{s=0}^{k-1} G_s^{-1 }(\textbf Z_{s})\bigg] \nonumber \\
\nonumber
\geq 
    &   \mathbb E \bigg[ \epsilon_{k-1}(\mbf Z_{k-1})\prod_{s=0}^{k-1} G^{*}_s (\textbf Z_{s})\bigg] 
     \mathbb E \left[ \frac{1} {\epsilon_{k-1}(\mbf Z_{k-1}) }\prod_{s=0}^{k-1} G^*_s (\textbf Z_{s})\right] 
\end{align}
and using the Cauchy-Schwarz inequality on the right term, we get that
 \begin{align}
 \mathbb E &\bigg[ \ \prod_{s=0}^{k-1} G_s (\textbf Z_{s})\bigg]  \mathbb E \bigg[ \mathbb E[h(\textbf Z_{n})|\textbf Z_{k} ]^2\, \prod_{s=0}^{k-1} G_s^{-1 }(\textbf Z_{s})\bigg] \nonumber \\
& \geq 
       \mathbb E \bigg[ \prod_{s=0}^{k-1} G^{*}_s (\textbf Z_{s})\bigg] ^2
    = \mathbb E \bigg[ \prod_{s=0}^{k-1} G^{*}_s (\textbf Z_{s})\bigg] \mathbb E \left[ \frac{g_{k-1}  ( \textbf Z_{k-1} )}{ \prod_{s=0}^{k-1} G^*_s (\textbf Z_{s})}\right]  \label{eq:bhvu}.
\end{align}
By summing the inequalities \eqref{eq:bhvu} with respect to $k$, we easily see that 
$$
\sigma^2_{IPS,G}\geq \sigma^2_{IPS,G^*}.
$$
The equality is achieved when $\epsilon_{k-1}(\mbf Z_{k-1})$ is constant $\nu_k$-a.s..
This completes the proof of the theorem.
\end{proof}

If  the optimal potential functions $G^*_k$ are positive-valued (i.e., if the functions $g_k$ defined by (\ref{def:gk}) are positive valued), 
then they satisfy the hypothesis under which the TCL was proven and
they can be used as potential functions for the IPS estimation method  and they give an estimator with 
the minimal asymptotic variance (\ref{eq:varOpti}). This is the case for all the examples addressed in the next section.
If the optimal potential functions $G^*_k$  are not positive everywhere, then they violate the hypothesis under which the TCL was proven in \cite{del2004feynman}. 
One could revisit the proof to extend its validity, in order to get a central limit theorem with weaker conditions \cite{chopin04,kunsch05}.
Here we simply state that, if the optimal potential functions $G^*_k$  are not positive everywhere,
then Theorem \ref{thm3} can be understood in the following way: Eq.~(\ref{eq:varOpti}) gives the infimum value 
of the asymptotic variance of the IPS estimator amongst all acceptable potential functions.
This value can be reached by a minimizing sequence of potential functions that are of the form $G_k=\max (G_k^*,\varepsilon)$, $k\leq n-1$, for $\varepsilon \searrow 0$.

\section{Numerical illustrations}
\label{sec:applica}
In this section we apply the IPS method to toy models. The first model under consideration  in section  \ref{subseq:second} is a Gaussian random walk,
that has been studied several times in the literature \cite{del2005genealogical,jacquemart2016tuning}. 
In section \ref{subseq:third} we consider an Ornstein-Uhlenbeck process.
In sections \ref{subseq:fifth}-\ref{subsec:sixth} we consider a more complex diffusion process.

\subsection{First example}\label{subseq:second}
We consider the Gaussian random walk $Z_{k+1} = Z_{k} + \varepsilon_{k+1}$,
$Z_0 = 0$, where the $\varepsilon_{k}$ are i.i.d. Gaussian random variables with mean zero
and variance one.
The goal is to compute the probability that $Z_n$ exceeds a large  positive value $a$.
Therefore we take $h(\mbf z_n)=\indic{[a,+\infty)}(z_n)$ so that
 \begin{equation}p_h =\mathbb P(Z_n\geq a).
 \end{equation}
In this case $ \mathbb E[h(\mbf Z_n)|\mbf Z_k=\mbf z_k]$ can be expressed in terms of the cdf $\Phi$ of the standard normal distribution:
\begin{align*}
\mathbb E[h(\mbf Z_n)| \mbf Z_k=\mbf z_{k } ] 
=  \Phi\Big( \frac{z_{k}-a}{\sqrt{n-k}}\Big) 
\end{align*}
The optimal potential functions satisfy
\begin{align*}
\nonumber
  \prod_{i=0}^{k-1}G_i^*(\mbf z_i)^2 &= \mathbb E\Big[\, \mathbb E\big[ h(\mbf Z_{n})\big|\mbf Z_{k}\big]^2\big|\mbf Z_{k-1} = \mbf z_{k-1} \Big] \\
  & =\Psi \big( z_{k-1}-a  ; n-k\big)
\end{align*}
where 
\begin{align}
\Psi(z;\rho) &= \frac{1}{\sqrt{2\pi}} \int_{\mathbb{R}} \Phi\Big( \frac{z -z'}{\sqrt{\rho}} \Big)^2 \exp\big( - \frac{{z'}^2}{2} \big) dz'.
\end{align}

Therefore, the optimal potential functions  are given by
\begin{align}
&
G_0^*(\mbf z_0)= \Psi \big( z_{0}-a ; n-1\big)^{1/2},
\label{def:Gstar0}
\end{align}
and for $1\leq k \leq n-1$
\begin{align}
&
G_k^*(\mbf z_k)=
\frac{
\Psi \big(  z_{k}-a ; n-k-1\big)^{1/2} 
}
{
\Psi \big( z_{k-1}-a ; n-k\big)^{1/2}
}
.
\label{def:Gstark}
\end{align}
The Chernov-Bernstein's inequality   
gives the rather 
sharp exponential bound: 
\begin{align*}
\mathbb E[h(\mbf Z_n)| \mbf Z_k=\mbf z_{k } ] 
 \leq \exp\left[-\frac{(a-z_{k })_+^2}{2(n-k)}\right],
\end{align*}
from which we can deduce that: 
\begin{align*}
\nonumber
  \prod_{i=0}^{k-1}G_i^*(\mbf z_i)^2 
&  \leq
 \exp\left[-\frac{(a-z_{k-1})_+^2}{2}\right]
+
\exp\left[- \frac{\left(a-z_{k-1}\right)^2}{n-k+2} \right]  .
\end{align*} 
By taking inspiration of this upper bound, we can propose to use the following explicit potential functions 
that turn out to be good approximations of the exact optimal potential functions $G^*_k$: 
 for $1\leq k \leq n-1$,
\begin{align}
G_k(\mbf z_k)& = \exp\left[- \frac{\left(z_{k}-a\right)^2}{2(n-k+1)}+ \frac{\left(z_{k-1}-a\right)^2}{2(n-k+2)} \right] 
\label{eq:Gex2a}
\end{align}
and
\begin{align}
G_0(\mbf z_0) & = \exp\left[- \frac{\left(z_{0}-a\right)^2}{2(n+1)}\right].
\label{eq:Gex2b} 
\end{align} 


In order to confirm the theoretical predictions, we have carried out a simulation study. We have run the method $1000$ times with $N=2 000$, $n=10$, different values of $a$, and different potential functions.
For each case we have computed the empirical mean and variance of the estimations. 
We have first used the potentials defined in equations (\ref{eq:Gex2a}-\ref{eq:Gex2b}). The results are displayed in table \ref{tab:res3}. We compare  these estimations  with the actual values of $p_h$ and the asymptotic variance of the Monte-Carlo method
which is simply $p_h$, showing that the potentials (\ref{eq:Gex2a}-\ref{eq:Gex2b}) built with the Chernov-Bernstein large deviation inequality yield a significant variance reduction. Indeed the variance reduction  compared to the Monte-Carlo method is at least by a factor $500$ and  at  best by a factor $5 \, 10^8$.

\begin{table*}[ht]
\centering
{\small
\begin{tabular}{|c|c|c|c|c|}
\toprule
\midrule
$a$  &$p_h$  &  $\sigma^2_{MC}$ & 
$mean(\hat p_h)$ & $\hat \sigma^2_{IPS,G}$ \\
\midrule
 $4\sqrt{n}$& $3.2 \, 10^{-5\ }$  & $3.2 \, 10^{-5\ }$ & 
 $3.2 \, 10^{-5\ }$ & $6.5 \, 10^{-8\ }$\\
 $5\sqrt{n}$& $2.9 \, 10^{-7\ }$ & $2.9 \, 10^{-7\ }$ & 
 $2.9 \, 10^{-7\ }$ & $1.3 \, 10^{-11}$\\
 $6\sqrt{n}$& $9.9 \, 10^{-10}$ & $9.9 \, 10^{-10}$ & 
 $1.0 \, 10^{-9}$ & $7.4 \, 10^{-16}$\\
 $7\sqrt{n}$& $1.3 \, 10^{-12}$  & $1.3 \, 10^{-12}$ & 
 $1.2 \, 10^{-12}$ & $2.6 \, 10^{-21}$\\
 \midrule
 \bottomrule
\end{tabular} 
}
\caption{Theoretical and empirical comparisons (example 1) with the potentials (\ref{eq:Gex2a}-\ref{eq:Gex2b}). The
results are obtained with $N=2 000$ and $n=10$.  The IPS algorithm has been run $1000$ times to get the empirical values
$mean(\hat p_h)$ and $\hat \sigma^2_{IPS,G^*}$.}
\label{tab:res3}
\end{table*} 

We have also compared the efficiencies of different potentials.
We have run the method $1000$ times with $N=2 000$, $n=10$, $a=15$, and different potential functions:\\
1)
the potential that selects the best values $ G_k(\mbf z_k)=\exp\left[- \alpha z_k \right]$ (we have determined that the optimal parameter $\alpha$ is  $\alpha=0.22$); \\
2) the potential used on a Gaussian random walk in \cite{jacquemart2016tuning} that selects the particles with the best increments: $ G_k(\mbf z_k)=\exp\left[\alpha(z_k-z_{k-1}) \right]$ 
(we have determined that the optimal parameter $\alpha$ is  $\alpha=1.4$);\\ 
 3) the potential (\ref{eq:Gex2a}-\ref{eq:Gex2b}) built with the Chernov-Bernstein large deviation inequality; \\
 4) the optimal  potential $(G_k^*)_{k<n}$ defined by (\ref{def:Gstar0}-\ref{def:Gstark})
 that we can compute using Gauss-Hermite quadrature formulas. \\
 The results  displayed in table \ref{tab:res4} show that \\
 1) By comparing $mean(\hat p_h)$  with $p_h$, we can see that  the estimator is unbiased whatever the choice of the potential.\\
2) By comparing the empirical variance multiplied by $N$ of the IPS estimator $\hat{\sigma}^2_{IPS,G}$, which is close to the asymptotic variance of the IPS estimator $\sigma^2_{IPS,G}$, to the asymptotic variance of the standard Monte Carlo estimator $\sigma^2_{MC}$, which is simply $p_h$, 
we can see that the IPS method provides a significant variance reduction, as the variance can be reduced by a factor $10^4$ with the optimal potential. \\
3) By comparing the two potentials parameterized with $\alpha$, 
we conclude that it is important to take into account the increments in the selection process and not only the values, because the minimal variance obtained with a potential of the form $ G_k(\mbf z_k)=\exp\left[- \alpha z_k \right]$ (the minimal variance is obtained with $\alpha=0.22$)
is $17$ times larger than the minimal variance obtained with a potential of the form $ G_k(\mbf z_k)=\exp\left[\alpha(z_k-z_{k-1}) \right]$
(the minimal variance is obtained with $\alpha=1.4$).\\
4)
The optimal potential functions $(G_k^*)_{k<n}$
 yield the best variance but the potential (\ref{eq:Gex2a}-\ref{eq:Gex2b}) 
 has almost the same performance, and the potential $ G_k(\mbf z_k)=\exp\left[\alpha(z_k-z_{k-1}) \right]$ with the tuned parameter $\alpha=1.4$
 has also a similar performance.

 \begin{table}[ht]
\centering
{\small
\begin{tabular}{|c|c|c|}
\toprule
\midrule
 $G_k(\mbf z_k)$  & $mean(\hat p_h)$ & $\hat \sigma^2_{IPS,G}$\\
\midrule
 $\exp\left[\alpha z_k \right]$, $\alpha=0.14$
    & $1.05\, 10^{-6}$ & $5.6\, 10^{-9}$\\
 $\exp\left[\alpha z_k \right]$, $\alpha=0.16$
    & $1.03\, 10^{-6}$ & $3.7\, 10^{-9}$\\
 $\exp\left[\alpha z_k \right]$, $\alpha=0.18$
    & $1.03\, 10^{-6}$ & $3.5\, 10^{-9}$\\
 $\exp\left[\alpha z_k \right]$, $\alpha=0.20$
    & $1.08\, 10^{-6}$ & $2.9\, 10^{-9}$\\
    $\exp\left[\alpha z_k \right]$, ${\bf \boldsymbol{\alpha} =0.22}$
    & $1.03\, 10^{-6}$ & ${\bf 2.8\, 10^{-9}}$\\
 $\exp\left[\alpha z_k \right]$, $\alpha=0.24$
    & $1.08\, 10^{-6}$ & $3.3\, 10^{-9}$\\
 $\exp\left[\alpha z_k \right]$, $\alpha=0.26$
    & $1.09\, 10^{-6}$ & $ 3.6  \, 10^{-9}$\\
  $\exp\left[\alpha z_k \right]$, $\alpha=0.28$
    & $1.05\, 10^{-6}$ & $4.0\, 10^{-9}$\\
 $\exp\left[\alpha(z_k\text{-}z_{k\text{-}1}) \right]$, $\alpha=1.0$
    & $1.05\, 10^{-6}$ & $3.2\, 10^{-10}$\\
 $\exp\left[\alpha(z_k\text{-}z_{k\text{-}1}) \right]$, $\alpha=1.1$
    & $1.06\, 10^{-6}$ & $2.5\, 10^{-10}$\\
 $\exp\left[\alpha(z_k\text{-}z_{k\text{-}1}) \right]$, $\alpha=1.2$
    & $1.04\, 10^{-6}$ & $1.9\, 10^{-10}$\\
     $\exp\left[\alpha(z_k\text{-}z_{k\text{-}1}) \right]$, $\alpha=1.3$
    & $1.05\, 10^{-6}$ & $1.7\, 10^{-10}$\\
     $\exp\left[\alpha(z_k\text{-}z_{k\text{-}1}) \right]$, ${\bf \boldsymbol{\alpha} =1.4}$
    & $1.05\, 10^{-6}$ & ${\bf 1.7\, 10^{-10}}$\\
     $\exp\left[\alpha(z_k\text{-}z_{k\text{-}1}) \right]$, $\alpha=1.5$
    & $1.05\, 10^{-6}$ & $1.8\, 10^{-10}$\\
     $\exp\left[\alpha(z_k\text{-}z_{k\text{-}1}) \right]$, $\alpha=1.6$
    & $1.05\, 10^{-6}$ & $2.0\, 10^{-10}$\\
     $\exp\left[\alpha(z_k\text{-}z_{k\text{-}1}) \right]$, $\alpha=1.7$
    & $1.05\, 10^{-6}$ & $2.4\, 10^{-10}$\\
  $G_k(\mbf z_k)   $  Eqs.~(\ref{eq:Gex2a}-\ref{eq:Gex2b})
& $1.05 \, 10^{-6}$ & ${\bf 1.5 \, 10^{-10}}$\\
  $G_k^*(\mbf z_k)   $  Eqs.~(\ref{def:Gstar0}-\ref{def:Gstark})
    & $1.05 \, 10^{-6}$ & ${\bf 1.3 \,10^{-10}}$\\
 \midrule
 \bottomrule
\end{tabular} 
}
\caption{Comparisons of the efficiencies of different potentials (example 1).
The results are obtained for $N=2000$,  $n=10$, $a=15$ (so $p_h=1.05 \, 10^{-6}$ and $\sigma^2_{MC}=1.05 \, 10^{-6}$).
 The IPS algorithm has been run $1000$ times to get the empirical values
$mean(\hat p_h)$ and $\hat \sigma^2_{IPS,G}$.
}
\label{tab:res4}
\end{table} 

\subsection{Second example}\label{subseq:third}
Here we consider the Ornstein-Uhlenbeck process, that is the solution of the stochastic differential equation:
 \begin{equation}
dX_t =\sqrt{2} dW_t - X_t dt , \quad \quad X_0=0.
 \end{equation}
It is a Markov, Gaussian, mean-reverting, and ergodic process.
The stationary distribution is ${\cal N}(0,1)$ and the mean-reversion time is $1$.
If the process is observed at times $t_k = k \tau$, $k=0,\ldots,n$, with $\tau>0$, then the observed process $Z_k=X_{t_k}$ is 
a homogeneous Markov chain
with kernel
 \begin{align}
 \nonumber
Q(dz_k | z_{k-1}) =& \frac{1}{\sqrt{2\pi ( 1-e^{-2\tau})}} 
\exp\Big[ - \frac{(z_k - z_{k-1}e^{-\tau} )^2}{2( 1-e^{-2\tau})}\Big] dz_k  ,
 \end{align}
in other words the law of $ Z_{k}$ knowing $(Z_j)_{j\leq k-1}$ is Gaussian with mean $ Z_{k-1} e^{-\tau}$ and variance $1-e^{-2\tau} $.
The goal is to compute the probability that $X_{t_n}$ exceeds a large  positive value $a$. Therefore we take $h(\mbf z_n)=\indic{[a,+\infty)}(z_n)$ so that
$p_h =\mathbb P(Z_n\geq a).$
 In this case we have for $0\leq k\leq n-1$:
 \begin{align*}
\mathbb E[h(\textbf Z_n)| \mbf Z_k= \mbf z_k] &=
\frac{1}{\sqrt{2\pi(1- e^{ -  2(n\text{-}k) \tau})}}
\int_a^\infty \exp\left[\text{-}\frac{(z_{n}'\text{-} z_{k}e^{\text{-} (n\text{-}k) \tau})^2}{2(1- e^{ -  2(n\text{-}k) \tau})}\right]dz_n' \\
&= \Phi\Big( \frac{z_{k}e^{\text{-} (n\text{-}k) \tau} \text{-} a}{\sqrt{1- e^{ -  2(n\text{-}k) \tau}}}\Big) ,
\end{align*}
 and the optimal potential is: for $1\leq k \leq n-2$, 
 \begin{align}
&
G_k^*(\mbf z_k)=
\sqrt{\frac{\int_{\mathbb R} 
\Phi\Big( \frac{ z_{k+1}' e^{\text{-} (n\text{-}k\text{-}1) \tau} \text{-} a}{\sqrt{1- e^{ -  2(n\text{-}k\text{-}1) \tau}}}\Big)
^2\exp\left[\text{-}\frac{(z_{k+1}'\text{-} z_{k}e^{- \tau})^2}{2(1- e^{- 2 \tau})}\right]dz_{k+1}'  }
{\int_{\mathbb R} \Phi\Big( \frac{ z_{k}'e^{\text{-} (n\text{-}k) \tau} \text{-} a}{\sqrt{1- e^{ -  2(n\text{-}k) \tau}}}\Big)^2\exp\left[\text{-}\frac{(z_{k}'\text{-} z_{k\text{-}1}e^{- \tau})^2}{2(1- e^{- 2 \tau})}\right]dz_{k}' }}  ,
\label{eq:Gkstaroua}
\end{align}
and
 \begin{align}
&
G_{n-1}^*(\mbf z_{n-1})=
\sqrt{\frac{\int_{\mathbb R} 
{\bf 1}_{[a,+\infty)}( z_{n}') \exp\left[\text{-}\frac{(z_{n}'\text{-} z_{n-1}e^{- \tau})^2}{2(1- e^{- 2 \tau})}\right]dz_{n}'  }
{\int_{\mathbb R} \Phi\Big( \frac{ z_{n-1}'e^{\text{-} \tau} \text{-} a}{\sqrt{1- e^{ -  2 \tau}}}\Big)^2\exp\left[\text{-}\frac{(z_{n-1}'\text{-} z_{n\text{-}2}e^{- \tau})^2}{2(1- e^{- 2 \tau})}\right]dz_{n-1}' }}  ,
\label{eq:Gkstaroub}
\end{align}
 \begin{align}
&
G_0^*(\mbf z_0)=
\sqrt{ \int_{\mathbb R} 
\Phi\Big( \frac{ z_{1}' e^{\text{-} (n\text{-}1) \tau} \text{-} a}{\sqrt{1- e^{ -  2(n\text{-}1) \tau}}}\Big)
^2\exp\left[\text{-}\frac{(z_{1}'\text{-} z_{0}e^{- \tau})^2}{2(1- e^{- 2 \tau})}\right]dz_{1}'  }  .
\label{eq:Gkstarouc}
\end{align}

The Chernov-Bernstein's inequality   
gives the rather 
sharp exponential bound: 
\begin{align*}
\mathbb E[h(\mbf Z_n)| \mbf Z_k=\mbf z_{k } ] 
 \leq \exp\left[-\frac{(a-z_{k }e^{-(n-k)\tau})_+^2}{2(1-e^{-2(n-k)\tau})}\right].
\end{align*}
By taking inspiration of this upper bound, we can propose to use the following explicit potential functions 
that turn out to be good approximations of the exact optimal potential functions $G^*_k$: 
 for $1\leq k \leq n-1$,
\begin{align}
\nonumber
G_k(\mbf z_k) =& \exp\left[- \frac{\left(z_{k}e^{-(n-k)\tau} -a\right)^2}{2(1+e^{-2(n-k-1)\tau} - 2 e^{-2(n-k)\tau)} )}
\right.\\ &\left.
 + \frac{\left(z_{k-1}e^{-(n-k+1)\tau}-a\right)^2}{2(1+e^{-2(n-k)\tau}-2e^{-2(n-k+1)\tau})} \right] 
\label{eq:Gex2aou}
\end{align}
and
\begin{align}
G_0(\mbf z_0) & = \exp\left[- \frac{\left(z_{0}e^{-n\tau} -a\right)^2}{2(1+e^{-2(n-1)\tau} - 2 e^{-2n\tau})}\right].
\label{eq:Gex2bou} 
\end{align} 
The form of the quasi-optimal potential $G_k$ is instructive: we can observe that, as long as $(n-k)\tau \gg 1$, the potential is constant, which means that there is no selection in the early steps of the dynamics.
Selection is effective when the remaining time $(n-k)\tau$ is of order one, that is the mean-reversion time of the process $X_t$.
This result is actually general and can be extended to all ergodic, mean-reverting processes. Indeed, since $\mathbb E[h(\mbf Z_n)| \mbf Z_k=\mbf z_{k } ] $ is constant (and approximately equal to $\mathbb E[h(\mbf Z_n)]$) when the time difference between $k$ and $n$ is larger than the mean-reversion time, the selection pressure of the optimal selection scheme is low in the early steps. It is only high when the remaining time
 is of the order of the mean-reversion time.
This also explains why the potential $G_k(\mbf z_k) = \exp[ - \alpha(z_k-z_{k-1})]$ is reasonably efficient (i.e. much better than standard Monte Carlo) but not as efficient as the optimal potential for these cases (see below).

We compare the efficiencies of different potentials in tables \ref{tab:res5}-\ref{tab:res5c}  for different values of $t_n$.
It appears that the quasi-optimal potential (\ref{eq:Gex2aou}-\ref{eq:Gex2bou}) has always a performance that is close to the optimal potential $G_k^*$ defined by (\ref{eq:Gkstaroua}-\ref{eq:Gkstarouc}).
When $t_n$ is smaller than one, then the mean-reversion force is low and the situation is close to the random Gaussian walk. The potential $G_k(\mbf z_k) = \exp[ - \alpha(z_k-z_{k-1})]$ with the optimal $\alpha$ ($\alpha=3.5$ for $t_n=1$) has a performance that is not far from the one of the optimal potential $G_k^*$.
When $t_n$ is larger than one, then the mean reversion is strong and an efficient selection is concentrated on the last time steps.
This is what the optimal potential $G_k^*$ and the quasi-optimal potential (\ref{eq:Gex2aou}-\ref{eq:Gex2bou}) do and this explains why the variance reduction is then very strong. The potential $G_k(\mbf z_k) = \exp[ - \alpha(z_k-z_{k-1})]$ with the optimal $\alpha$ ($\alpha=2.5$ for $t_n=2$ and $\alpha=2$ for $t_n=3$) imposes a constant selection pressure throughout the history of the Markov process, it does much better than Monte Carlo
(by a factor $50$--$100$ in terms of variance), but not as good as the optimal or quasi-optimal potentials (by a factor $2$--$3$ in terms of variance). 

\begin{table}[ht]
\centering
{\small
\begin{tabular}{|c|c|c|c|c|}
\toprule
\midrule
$G_k(\mbf z_k)$  & $mean(\hat p_h)$ & $\hat \sigma^2_{IPS,G}$\\
\midrule
 $\exp\left[\alpha z_k \right]$, $\alpha=0.1$
    & $8.6\, 10^{-6}$ & $2.3\, 10^{-6}$\\
 $\exp\left[\alpha z_k \right]$, $\alpha=0.2$
    & $8.4\, 10^{-6}$ & $7.7\, 10^{-7}$\\
 $\exp\left[\alpha z_k \right]$, $\alpha=0.3$
    & $8.6\, 10^{-6}$ & $4.7\, 10^{-7}$\\
 $\exp\left[\alpha z_k \right]$, ${\bf \boldsymbol{\alpha} = 0.4}$
    & $8.6 \, 10^{-6}$ & ${\bf 3.2\, 10^{-7}}$\\
 $\exp\left[\alpha z_k \right]$, $\alpha=0.5$
    & $8.2\, 10^{-6}$ & $3.5\, 10^{-7}$\\
 $\exp\left[\alpha z_k \right]$, $\alpha=0.6$
    & $8.1 \, 10^{-6}$ & $5.4  \, 10^{-7}$\\
 $\exp\left[\alpha z_k \right]$, $\alpha=0.7$
    & $8.1 \, 10^{-6}$ & $8.7  \, 10^{-7}$\\
 $\exp\left[\alpha z_k \right]$, $\alpha=0.8$
    & $7.8 \, 10^{-6}$ & $1.5\, 10^{-6}$\\
 $\exp\left[\alpha(z_k\text{-}z_{k\text{-}1}) \right]$, $\alpha=2.0$
    & $8.6  \, 10^{-6}$ & $5.2 \, 10^{-8}$\\
 $\exp\left[\alpha(z_k\text{-}z_{k\text{-}1}) \right]$, $\alpha=2.5$
    & $8.5  \, 10^{-6}$ & $2.8 \, 10^{-8}$\\
 $\exp\left[\alpha(z_k\text{-}z_{k\text{-}1}) \right]$, $\alpha=3.0$
    & $8.4 \, 10^{-6}$ & $2.0 \, 10^{-8}$\\
     $\exp\left[\alpha(z_k\text{-}z_{k\text{-}1}) \right]$, ${\bf \boldsymbol{\alpha}=3.5}$
    & $8.5\, 10^{-6}$ & ${\bf 1.8  \, 10^{-8}}$\\
 $\exp\left[\alpha(z_k\text{-}z_{k\text{-}1}) \right]$, $\alpha=4.0$
    & $8.5\, 10^{-6}$ & $2.5 \, 10^{-8}$\\
 $\exp\left[\alpha(z_k\text{-}z_{k\text{-}1}) \right]$, $\alpha=4.5$
    & $8.5\, 10^{-6}$ & $3.8  \, 10^{-8}$\\
 $\exp\left[\alpha(z_k\text{-}z_{k\text{-}1}) \right]$, $\alpha=5.0$
    & $8.5\, 10^{-6}$ & $1.0  \, 10^{-7}$\\
 $G_k(\mbf z_k)$  Eqs.~(\ref{eq:Gex2aou}-\ref{eq:Gex2bou})
    &  $8.5\, 10^{-6}$ & ${\bf 1.5\, 10^{-8}}$\\
  $G_k^*(\mbf z_k)   $  Eqs.~(\ref{eq:Gkstaroua}-\ref{eq:Gkstarouc})
    &  $8.5 \, 10^{-6}$ & ${\bf 1.4 \,10^{-8}}$\\
 \midrule
 \bottomrule
\end{tabular} 
}
\caption{Comparisons of the efficiencies of potentials (example 2).
The results are obtained for  $N=2000$,  $a=4$, $\tau=0.1$, $n=10$. Here $t_n=1$ and $p_h=8.5\, 10^{-6}$ (so $\sigma^2_{MC}= 8.5\, 10^{-6} $).
 The IPS algorithm has been run $1000$ times to get the empirical values
$mean(\hat p_h)$ and $\hat \sigma^2_{IPS,G}$.
}
\label{tab:res5}
\end{table}

\begin{table}[ht]
\centering
{\small
\begin{tabular}{|c|c|c|}
\toprule
\midrule
$G_k(\mbf z_k)$  & $mean(\hat p_h)$ & $\hat \sigma^2_{IPS,G}$\\
\midrule
 $\exp\left[\alpha z_k  \right]$, $\alpha=0.05$
    & $2.6 \, 10^{-5}$ & $1.0 \, 10^{-5}$\\
 $\exp\left[\alpha z_k  \right]$, $\alpha=0.10$
    & $2.6 \, 10^{-5}$ & $5.6\, 10^{-6}$\\
 $\exp\left[\alpha z_k  \right]$, ${\bf \boldsymbol{\alpha}=0.15}$
    & $2.7 \, 10^{-5}$ & ${\bf 4.3 \, 10^{-6}}$\\
 $\exp\left[\alpha z_k  \right]$, $\alpha=0.20$
    & $2.7 \, 10^{-5}$ & $ 5.1 \, 10^{-6}$\\
 $\exp\left[\alpha z_k  \right]$, $\alpha=0.25$
    & $2.7 \, 10^{-5}$ & $1.9  \, 10^{-5}$\\
 $\exp\left[\alpha z_k  \right]$, $\alpha=0.30$
    & $2.8 \, 10^{-5}$ & $5.2  \, 10^{-5}$\\
 $\exp\left[\alpha(z_k\text{-}z_{k\text{-}1}) \right]$, $\alpha=1.0$
    & $2. 8\, 10^{-5}$ & $ 1.4\, 10^{-6}$\\
 $\exp\left[\alpha(z_k\text{-}z_{k\text{-}1}) \right]$, $\alpha=1.5$
    & $2. 8\, 10^{-5}$ & $4.5 \, 10^{-7}$\\
 $\exp\left[\alpha(z_k\text{-}z_{k\text{-}1}) \right]$, $\alpha=2.0$
    & $2.6\, 10^{-5}$ & $ 2.1 \, 10^{-7}$\\
 $\exp\left[\alpha(z_k\text{-}z_{k\text{-}1}) \right]$, ${\bf \boldsymbol{\alpha}=2.5}$
    & $2.7\, 10^{-5}$ & ${\bf 1.4  \, 10^{-7}}$\\
 $\exp\left[\alpha(z_k\text{-}z_{k\text{-}1}) \right]$, $\alpha=3.0$
    & $2.7 \, 10^{-5}$ & $1.9 \, 10^{-7}$\\
 $\exp\left[\alpha(z_k\text{-}z_{k\text{-}1}) \right]$, $\alpha=3.5$
    & $2.7 \, 10^{-5}$ & $4.6 \, 10^{-7}$\\
 $\exp\left[\alpha(z_k\text{-}z_{k\text{-}1}) \right]$, $\alpha=4.0$
    & $ 2.8 \, 10^{-5}$ & $ 2.4 \, 10^{-6}$\\
 $G_k(\mbf z_k)$  Eqs.~(\ref{eq:Gex2aou}-\ref{eq:Gex2bou})
    &  $2.7 \, 10^{-5}$ & ${\bf 7.7 \, 10^{-8}}$\\
   $G_k^*(\mbf z_k)   $   Eqs.~(\ref{eq:Gkstaroua}-\ref{eq:Gkstarouc})
    &  $2.7 \, 10^{-5}$ & ${\bf 7.5 \,10^{-8}}$\\
 \midrule
 \bottomrule
\end{tabular} 
}
\caption{Comparisons of the efficiencies of potentials (example 2).
The results are obtained for  $N=2000$,  $a=4$, $\tau=0.1$, $n=20$. Here $t_n=2$ and $p_h= 2.7\, 10^{-5} $ (so $\sigma^2_{MC}= 2.7\, 10^{-5} $).
 The IPS algorithm has been run $1000$ times to get the empirical values
$mean(\hat p_h)$ and $\hat \sigma^2_{IPS,G}$.
}
\label{tab:res5b}
\end{table}

\begin{table}[ht]
\centering
{\small
\begin{tabular}{|c|c|c|}
\toprule
\midrule
$G_k(\mbf z_k)$  & $mean(\hat p_h)$ & $\hat \sigma^2_{IPS,G}$\\
\midrule
 $\exp\left[\alpha z_k  \right]$, $\alpha=0.05$
    & $3.2 \, 10^{-5}$ & $1.4 \, 10^{-5}$\\
 $\exp\left[\alpha z_k  \right]$, ${\bf \boldsymbol{\alpha}=0.10}$
    & $3.0 \, 10^{-5}$ & ${\bf 1.3\, 10^{-5}}$\\
 $\exp\left[\alpha z_k  \right]$, $\alpha=0.15$
    & $3.1 \, 10^{-5}$ & $2.2 \, 10^{-5}$\\
 $\exp\left[\alpha z_k  \right]$, $\alpha=0.20$
    & $2.7 \, 10^{-5}$ & $ 5.8 \, 10^{-5}$\\
 $\exp\left[\alpha(z_k\text{-}z_{k\text{-}1}) \right]$, $\alpha=1.0$
    & $3.1\, 10^{-5}$ & $ 1.6\, 10^{-6}$\\
 $\exp\left[\alpha(z_k\text{-}z_{k\text{-}1}) \right]$, $\alpha=1.5$
    & $3.1\, 10^{-5}$ & $5.5 \, 10^{-7}$\\
 $\exp\left[\alpha(z_k\text{-}z_{k\text{-}1}) \right]$, ${\bf \boldsymbol{\alpha}=2.0}$
    & $3.1\, 10^{-5}$ & $ {\bf 3.1 \, 10^{-7}}$\\
 $\exp\left[\alpha(z_k\text{-}z_{k\text{-}1}) \right]$, $\alpha =2.5$
    & $3.1\, 10^{-5}$ & $4.6   \, 10^{-7}$\\
 $\exp\left[\alpha(z_k\text{-}z_{k\text{-}1}) \right]$, $\alpha=3.0$
    & $3.0 \, 10^{-5}$ & $1.7\, 10^{-6}$\\
 $G_k(\mbf z_k)$  Eqs.~(\ref{eq:Gex2aou}-\ref{eq:Gex2bou})
    &  $3.1 \, 10^{-5}$ & ${\bf 9.4 \, 10^{-8}}$\\
   $G_k^*(\mbf z_k)   $   Eqs.~(\ref{eq:Gkstaroua}-\ref{eq:Gkstarouc})
    &  $3.1 \, 10^{-5}$ & ${\bf 9.0 \,10^{-8}}$\\
 \midrule
 \bottomrule
\end{tabular} 
}
\caption{Comparisons of the efficiencies of potentials (example 2).
The results are obtained for  $N=2000$,  $a=4$, $\tau=0.1$, $n=30$. Here $t_n=3$ and $p_h= 3.1\, 10^{-5} $ (so $\sigma^2_{MC}= 3.1\, 10^{-5} $).
 The IPS algorithm has been run $1000$ times to get the empirical values
$mean(\hat p_h)$ and $\hat \sigma^2_{IPS,G}$.
}
\label{tab:res5c}
\end{table}

\subsection{Third example}
\label{subseq:fifth}%
We finally consider a diffusion process, that is the solution of the stochastic differential equation
\begin{equation}
\label{eq:sdeex}
dX_t = \sqrt[4]{1+X_t^2} dW_t - (X_t-0.5\sin(X_t)) dt, \quad \quad 
X_0=0.
\end{equation}
The discretized process $Z_k=X_{k\tau}$ is Markov.
The goal is to compute the probability that $X_{n\tau}$ exceeds a large  positive value $a$.
Therefore we take $h(\mbf z_n)=\indic{[a,+\infty)}(z_n)$ so that
$p_h =\mathbb P(Z_n\geq a)$.
There is no way to compute analytically $p_h$.
Table \ref{tab:res5d} shows that the IPS method  
can easily reduce the variance of the estimation of $p_h$ compared to the standard Monte Carlo method by a factor up to $7000$ when 
$a=6$ and $p_h \simeq 1.3 \, 10^{-7}$.

\subsection{Fourth example}
\label{subsec:sixth}%
We consider again the diffusion process (\ref{eq:sdeex}) and its discretized version $Z_k=X_{k\tau}$.
The goal is now to compute the probability that $V(X_{n\tau})$ exceeds a large  positive value $a$, where $V(x)=|x|$.
 Therefore we take $h(\mbf z_n)=\indic{[a,+\infty)}(V(z_n))$ so that
$p_h =\mathbb P(V(Z_n)\geq a)$ (of course, by obvious symmetry, $p_h$ is here equal to twice the value of $p_h$ in the previous section).
This situation could seem more tricky than the one addressed in the previous section, because there are apparently two modes
of large deviations, either the process $X_t$ takes large values or it takes small values.
Under such circumstances, an importance sampling strategy, 
if it could be implemented, would have to be calibrated carefully with a bimodal biased distribution.
The IPS method is very robust to this kind of difficulties and automatically allocates particles in both important regions.
Table \ref{tab:res5e} shows that the IPS method  
can easily reduce the variance of the estimation of $p_h$ compared to the standard Monte Carlo method by a factor up to $4000$ when 
$a=6$ and $p_h \simeq 2.5 \, 10^{-7}$.

\begin{table*}[ht]
\hspace*{-0.6in}
{\small
\begin{tabular}{|c|c|c|c|c|c|c|c|c|}
\toprule
\midrule
& \multicolumn{2}{c|}{$a=3$} & \multicolumn{2}{c|}{$a=4$} & \multicolumn{2}{c|}{$a=5$} & \multicolumn{2}{c|}{$a=6$}\\
\midrule
$\alpha$ & $mean(\hat p_h)$ & $\hat \sigma^2_{IPS,G}$& $mean(\hat p_h)$ & $\hat \sigma^2_{IPS,G}$& $mean(\hat p_h)$ & $\hat \sigma^2_{IPS,G}$& $mean(\hat p_h)$ & $\hat \sigma^2_{IPS,G}$\\
\midrule
$0.5$
    & $7.9\, 10^{-4}$ & $ 3.1\, 10^{-4}$  & $4.5 \, 10^{-5}$ &  $1.0 \, 10^{-5}$ &
    $2.7\, 10^{-6}$ & $4.0 \, 10^{-7} $ & $1.4 \, 10^{-7}$ & $1.3 \, 10^{-8}$ \\
 $1.0$
    & $7.8\, 10^{-4}$ & $ 1.2\, 10^{-4}$ & $4.4 \, 10^{-5}$ &  $2.6 \, 10^{-6}$ &   $2.3\, 10^{-6}$ & $5.6 \, 10^{-8} $ &  $1.6 \, 10^{-7}$ &  $1.6 \, 10^{-9}$ \\
 $1.5$
    & $7.9\, 10^{-4}$ & $5.8 \, 10^{-5}$ & $4.4 \, 10^{-5}$ &  $8.4 \, 10^{-7}$ & $2.3\, 10^{-6}$ &$1.3 \, 10^{-8} $ &  $1.3 \, 10^{-7}$ & $2.3 \, 10^{-10}$\\
$2.0$
    & $7.9\, 10^{-4}$ & $ 3.8 \, 10^{-5}$ & $4.4 \, 10^{-5}$ &$3.9 \, 10^{-7}$ & $2.3\, 10^{-6}$ & $4.4 \, 10^{-9} $ &  $1.3 \, 10^{-7}$ & $5.6 \, 10^{-11}$\\
${\bf 2.5}$
    & $7.9\, 10^{-4}$ & ${\bf 3.7 \, 10^{-5}}$ & $4.4 \, 10^{-5}$ & ${\bf 2.9 \, 10^{-7}}$ &$2.3\, 10^{-6}$ &${\bf 2.4 \, 10^{-9} }$ & $1.3 \, 10^{-7}$ & ${\bf 1.9 \, 10^{-11}}$ \\
 $3.0$
    & $7.9 \, 10^{-4}$ & $1.3 \, 10^{-4}$ & $4.4 \, 10^{-5}$ & $6.1 \, 10^{-7}$ & $2.3\, 10^{-6}$ & $4.0 \, 10^{-9} $ &  $1.3 \, 10^{-7}$ & $3.5 \, 10^{-11}$\\
$3.5$
    & $8.4 \, 10^{-4}$ & $6.2 \, 10^{-3}$ & $4.9 \, 10^{-5}$ &$4.6 \, 10^{-4}$ & $2.1\, 10^{-6}$ & $1.8 \, 10^{-8} $ &  $1.2 \, 10^{-7}$ & $1.6 \, 10^{-10}$\\
 $4.0$
    & $ 5.4 \, 10^{-4}$ & $ 1.9 \, 10^{-2}$ & $4.2 \, 10^{-5}$ &$4.9 \, 10^{-4}$ & $2.1\, 10^{-6}$ & $8.6 \, 10^{-7} $ & $1.2 \, 10^{-7}$ & $1.3 \, 10^{-9}$ \\
     \midrule
 \bottomrule
\end{tabular} 
}
\caption{
Comparisons of the efficiencies of the potentials $\exp\left[\alpha(z_k\text{-}z_{k\text{-}1}) \right]$ for different $\alpha$ 
in order to estimate $p_h = \mathbb{P}( Z_n \geq a)$ (example 3).
The results are obtained for  $N=2000$,  $a\in \{3,4,5,6\}$, $\tau=0.1$, $n=10$ ($t_n=1$). 
The SDE (\ref{eq:sdeex}) is solved by the Euler-Maruyama  method with a time step $10^{-3}$.
The IPS algorithm has been run $1000$ times to get the empirical values
$mean(\hat p_h)$ and $\hat \sigma^2_{IPS,G}$.
}
\label{tab:res5d}
\end{table*}

\begin{table*}[ht]
\hspace*{-0.6in}
{\small
\begin{tabular}{|c|c|c|c|c|c|c|c|c|}
\toprule
\midrule
& \multicolumn{2}{c|}{$a=3$} & \multicolumn{2}{c|}{$a=4$} & \multicolumn{2}{c|}{$a=5$} & \multicolumn{2}{c|}{$a=6$}\\
\midrule
$\alpha$ & $mean(\hat p_h)$ & $\hat \sigma^2_{IPS,G}$& $mean(\hat p_h)$ & $\hat \sigma^2_{IPS,G}$& $mean(\hat p_h)$ & $\hat \sigma^2_{IPS,G}$& $mean(\hat p_h)$ & $\hat \sigma^2_{IPS,G}$\\
\midrule
$0.5$
    & $1.6\, 10^{-3}$ & $ 8.6\, 10^{-4}$  & $9.0 \, 10^{-5}$ &  $2.7 \, 10^{-5}$ & $5.2\, 10^{-6}$ & $1.0 \, 10^{-6} $ & $2.8 \, 10^{-7}$ & $3.6 \, 10^{-8}$ \\
 $1.0$
    & $1.6\, 10^{-3}$ & $ 3.8\, 10^{-4}$ & $8.9 \, 10^{-5}$ &  $7.8 \, 10^{-6}$ &  $4.5\, 10^{-6}$ & $1.7 \, 10^{-7} $ &  $2.4 \, 10^{-7}$ &  $3.9 \, 10^{-9}$ \\
 $1.5$
    & $1.6\, 10^{-3}$ & $1.9 \, 10^{-4}$ & $8.9 \, 10^{-5}$ &  $3.0 \, 10^{-6}$ & $4.5\, 10^{-6}$ &$4.4 \, 10^{-8} $ &  $2.6 \, 10^{-7}$ & $9.7 \, 10^{-10}$\\
$2.0$
    & $1.6\, 10^{-3}$ & ${\bf 1.3 \, 10^{-4}}$ & $8.8 \, 10^{-5}$ &$1.5 \, 10^{-6}$ & $4.5\, 10^{-6}$ & $1.6 \, 10^{-8} $ &  $2.5 \, 10^{-7}$ & $1.6 \, 10^{-10}$\\
${\bf 2.5}$
    & $1.6\, 10^{-3}$ & $ {1.4 \, 10^{-4}}$ & $8.8 \, 10^{-5}$ & ${\bf 1.2 \, 10^{-6}}$ &$4.6\, 10^{-6}$ &${\bf 1.0 \, 10^{-8}}$ & $2.5 \, 10^{-7}$ & ${\bf 7.0 \, 10^{-11}}$ \\
 $3.0$
    & $1.6 \, 10^{-3}$ & $4.2 \, 10^{-4}$ & $8.8 \, 10^{-5}$ & $2.9 \, 10^{-6}$ & $4.5\, 10^{-6}$ & $1.8 \, 10^{-8} $ &  $2.6 \, 10^{-7}$ & $1.0 \, 10^{-10}$\\
$3.5$
    & $1.5 \, 10^{-3}$ & $8.7 \, 10^{-3}$ & $8.7 \, 10^{-5}$ &$3.7 \, 10^{-5}$ & $4.3\, 10^{-6}$ & $1.2 \, 10^{-7} $ &  $2.5 \, 10^{-7}$ & $7.7 \, 10^{-10}$\\
 $4.0$
    & $ 1.3 \, 10^{-3}$ & $ 1.4 \, 10^{-1}$ & $8.4 \, 10^{-5}$ &$7.5 \, 10^{-4}$ & $3.9\, 10^{-6}$ & $3.7 \, 10^{-6} $ & $2.4 \, 10^{-7}$ & $2.4 \, 10^{-8}$ \\
     \midrule
 \bottomrule
\end{tabular} 
}
\caption{
Comparisons of the efficiencies of the potentials $\exp\left[\alpha(V(z_k)\text{-}V(z_{k\text{-}1})) \right]$ in order to estimate 
$p_h = \mathbb{P}( V(Z_n) \geq a)$ for $V(z)=|z|$ and for different $\alpha$ (example 4).
The results are obtained for  $N=2000$,  $a\in \{3,4,5,6\}$, $\tau=0.1$, $n=10$ ($t_n=1$). 
The SDE (\ref{eq:sdeex}) is solved by the Euler-Maruyama  method with a time step $10^{-3}$.
The IPS algorithm has been run $1000$ times to get the empirical values
$mean(\hat p_h)$ and $\hat \sigma^2_{IPS,G}$.
}
\label{tab:res5e}
\end{table*}

\section{Discussion and conclusion}
\label{sec:disc}

In this paper we give the expressions of the   potential functions minimizing the asymptotic variance of the IPS method with multinomial resampling. The existence of optimal potential functions proves that the possible variance reduction of an IPS method is lower-bounded. The expressions  have been validated analytically, and the expression for the minimal variance has been empirically confirmed in numerical simulations.

The analysis of  the expressions of the optimal potential functions confirm empirical observations reported in the literature. 
First, in \cite{del2005genealogical} the authors make the observation that it seems better to build a potential that depends on the increments of an energy function, this observation is confirmed as the optimal potential function is the multiplicative increment of the quantity $\sqrt{\mathbb E\Big[\, \mathbb E\big[ h(\mbf Z_{n})\big|\mbf Z_{k}\big]^2\big|\mbf Z_{k-1} = \mbf z_k\Big]}$, which is then the optimal energy function. Second, the fact that in \cite{wouters2016rare} the authors find better results with time-dependent potentials is explained by the fact the expressions of the optimal potential functions show an explicit dependence on $k$. 
Third, one can notice that  the optimal potential function $G_k^*(\mathbf z_k) $ only depends on $z_k$,  $z_{k-1}$  and on $k$. 
This shows that there is no advantage in looking for potential functions that depend on other variables.  
Fourth, if the underlying Markov process is ergodic and mean-reverting, then an efficient potential does not impose any pressure selection in the early steps of the random dynamics but only in the final steps, during the time interval that preceeds the target time and that has a width 
of the order of the mean-reversion time.
Finally, as splitting methods can be viewed as a version of the IPS method with indicator potential functions, our results show  that the selections of splitting algorithms are not optimal, and could be improved by using information on the expectations $\mathbb E\big[ h(\mbf Z_{n})\big|\mbf Z_{k }=\textbf z_k\big]$.

The optimal potential functions may be hard to find in practice. Indeed, the expectations $\mathbb E\big[ h(\mbf Z_{n})\big|\mbf Z_{k }=\mbf z_k\big]$ play a big role in the expression of the optimal potentials, but if we are trying  to assess $p_h=\mathbb E\big[ h(\mbf Z_{n}) \big]$, we typically lack information about the conditional expectations $\mathbb E\big[ h(\mbf Z_{n})\big|\mbf Z_{k }=\mbf z_k \big]$. If no information on the expectation $\mathbb E\big[ h(\mbf Z_{n})\big|\mbf Z_{k }=\mbf z_k\big]$ is available, it might be preferable to use more naive variance reduction method, where no input functions are needed. In such context, the Weighted Ensemble (WE) method  \cite{aristoff2018analysis,aristoff2018arXiv180600860} seems to be a good candidate, as it does not take in input potential functions but only a partition of the state space.
Conversely if the practitioner has qualitative information about the expectation $\mathbb E\big[ h(\mbf Z_{n})\big|\mbf Z_{k }=\mbf z_k\big]$, this information could be used to derive  very efficient potentials. 
This information could be acquired through an adaptive version of the IPS algorithm.

We believe that a  favorable context for the optimal IPS method  is the multifidelity Monte Carlo framework \cite{multifidelity}. 
Here we typically consider a complex system whose simulation cost by a high-fidelity code is very large, 
but it may happen that a low-fidelity code is available
that is faster to run but approximate. Under such circumstances, the low-fidelity code could be used to estimate the optimal potential functions,
and then the IPS method could be applied with the high-fidelity code.

The (partial) knowledge of the expectations $\mathbb E\big[ h(\mbf Z_{n})\big|\mbf Z_{k }=\mbf z_k\big]$
is therefore important for a well optimized use of the IPS method, but it is  interesting to remark that the same knowledge seems to be important for a well optimized importance sampling method  \cite{chraibi2019optimal}. This confirms the well known fact that, with a good knowledge of the dynamics of the process $(\mbf Z_k)_{k\geq0}$, the importance sampling method may be preferable to the IPS if it can be implemented. 
Nonetheless the IPS method may still be preferred to the importance sampling method 
because 1) it is not intrusive and 2) it should be preferred when one fears to be in an over-biasing situation, because   the IPS method  does not alter the propagation and the over-biasing phenomenon should then be less important than with importance sampling methods.

\providecommand{\bysame}{\leavevmode\hbox to3em{\hrulefill}\thinspace}
\providecommand{\MR}{\relax\ifhmode\unskip\space\fi MR }
\providecommand{\MRhref}[2]{%
  \href{http://www.ams.org/mathscinet-getitem?mr=#1}{#2}
}
\providecommand{\href}[2]{#2}


\begin{thebibliography}{10}

\bibitem{aristoff2018analysis}
D.~Aristoff, \emph{Analysis and optimization of weighted ensemble sampling},
  ESAIM: Mathematical Modelling and Numerical Analysis {{52}} (2018),
  pp.~1219--1238.

\bibitem{aristoff2018arXiv180600860}
D.~{Aristoff}, \emph{Optimizing weighted ensemble sampling for steady states
  and rare events}, SIAM Multiscale Model. Simul. {18} (2020), pp.~646--673.

\bibitem{berg00}
B.~A. Berg, \emph{Introduction to {M}ulticanonical {M}onte {C}arlo
  simulations}, Fields Institute Communications {26} (2000), pp.~1--24.

\bibitem{cerou2011nonasymptotic}
F.~C{\'e}rou, P.~Del~Moral and A.~Guyader, \emph{A nonasymptotic theorem for
  unnormalized {F}eynman--{K}ac particle models}, Annales de l'Institut Henri
  Poincar{\'e}, Probabilit{\'e}s et Statistiques {{47}} (2011), pp.~629--649.

\bibitem{cerou2007adaptive}
F.~C{\'e}rou and A.~Guyader, \emph{Adaptive multilevel splitting for rare event
  analysis}, Stochastic Analysis and Applications {{25}} (2007), pp.~417--443.

\bibitem{cerou2005limit}
F.~C{\'e}rou, F.~LeGland, P.~Del~Moral and P.~Lezaud, \emph{Limit theorems for
  the multilevel splitting algorithm in the simulation of rare events}.
  Simulation Conference, Orlando, United States. Proceedings of the Winter,
  pp.~682--691, IEEE, dec 2005.

\bibitem{chan2013general}
H.P. Chan and T.L. Lai, \emph{A general theory of particle filters in hidden
  {M}arkov models and some applications}, Annals of Statistics {{41}} (2013),
  pp.~2877--2904.

\bibitem{chopin04}
N.~Chopin, \emph{Central limit theorem for sequential {M}onte {C}arlo methods
  and its applications to {B}ayesian inference}, Annals of Statistics {32}
  (2004), pp.~2385--2411.

\bibitem{chraibi2019optimal}
H.~Chraibi, A.~Dutfoy, T.~Galtier and J.~Garnier, \emph{On the optimal
  importance process for piecewise deterministic {M}arkov process}, ESAIM:
  Probability and Statistics {23} (2019), pp.~893--921.

\bibitem{del2004feynman}
P.~Del~Moral, \emph{Feynman-Kac Formulae, Genealogical and Interacting Particle
  Systems With Applications}. Springer, New York, 2004.

\bibitem{del2006sequential}
P.~Del~Moral, A.~Doucet and A.~Jasra, \emph{Sequential {M}onte-{C}arlo
  samplers}, Journal of the Royal Statistical Society: Series B (Statistical
  Methodology) {{68}} (2006), pp.~411--436.

\bibitem{del2005genealogical}
P.~Del~Moral and J.~Garnier, \emph{Genealogical particle analysis of rare
  events}, Annals of Applied Probability {{15}} (2005), pp.~2496--2534.

\bibitem{douc09}
R.~Douc, E.~Moulines and J.~Olsson, \emph{Optimality of the auxiliary particle
  filter}, Probability and Mathematical Statistics {29} (2009), pp.~1--28.

\bibitem{doucet01}
A.~Doucet, N.~de~Freitas and N.~Gordon, \emph{Sequential {M}onte {C}arlo
  Methods in Practice}. Springer, New York, 2001.

\bibitem{garnierdelmo06}
J.~Garnier and P.~Del~Moral, \emph{Simulations of rare events in fiber optics
  by interacting particle systems}, Opt. Commun. {267} (2006), pp.~205--214.

\bibitem{gerber19}
M.~Gerber, N.~Chopin and N.~Whiteley, \emph{Negative association, ordering and
  convergence of resampling methods}, Annals of Statistics {47} (2019),
  pp.~2236--2260.

\bibitem{guarniero2017iterated}
P.~Guarniero, A.~M. Johansen and A.~Lee, \emph{The iterated auxiliary particle
  filter}, Journal of the American Statistical Association {112} (2017),
  pp.~1636--1647.

\bibitem{jacquemart2016tuning}
D.~Jacquemart and J.~Morio, \emph{Tuning of adaptive interacting particle
  system for rare event probability estimation}, Simulation Modelling Practice
  and Theory {{66}} (2016), pp.~36--49.

\bibitem{kunsch01}
H.-R. K\"unsch, \emph{State space and hidden {M}arkov models}, Complex
  Stochastic Systems (D.~R.~Cox O.~E. Barndorff-Nielsen and C.~Kl\"uppelberg,
  eds.), Chapman and Hall, London, 2001, pp.~109--173.

\bibitem{kunsch05}
\bysame, \emph{Recursive {M}onte {C}arlo filters: {A}lgorithms and theoretical
  analysis}, Annals of Statistics {33} (2005), pp.~1983--2021.

\bibitem{lee2018variance}
A.~Lee and N.~Whiteley, \emph{Variance estimation in the particle filter},
  Biometrika {105} (2018), pp.~609--625.

\bibitem{morio2013optimisation}
J.~Morio, D.~Jacquemart, M.~Balesdent and J.~Marzat, \emph{Optimisation of
  interacting particle systems for rare event estimation}, Computational
  Statistics \& Data Analysis {{66}} (2013), pp.~117--128.

\bibitem{multifidelity}
B.~Peherstorfer, K.~Willcox and M.~Gunzburger, \emph{Survey of multifidelity
  methods in uncertainty propagation, inference, and optimization}, SIAM Review
  {60} (2018), pp.~550--591.

\bibitem{rubinstein81}
R.~Y. Rubinstein, \emph{Simulation and the {M}onte {C}arlo method}. Wiley, New
  York, 1981.

\bibitem{wouters2016rare}
J.~Wouters and F.~Bouchet, \emph{Rare event computation in deterministic
  chaotic systems using genealogical particle analysis}, Journal of Physics A:
  Mathematical and Theoretical {{49}} (2016), p.~374002.

\end{thebibliography}
\end{document}